\documentclass[12pt]{amsart}
\usepackage{amsmath,amsthm,epsfig,amssymb,amscd,amsfonts} 
\usepackage[all]{xy}
\newtheorem{thm}[subsection]{Theorem}
\newtheorem{prop}[subsection]{Proposition}
\newtheorem{lem}[subsection]{Lemma}

\theoremstyle{definition}
\newtheorem{Def}[subsection]{Definition}
\newtheorem{remark}[subsection]{Remark}
\newtheorem{example}[subsection]{Example}

\newcommand{\OOO}{{\mathcal O}}
\numberwithin{equation}{section}

\begin{document}

\title[Direct image of the adjoint line bundle]{On the direct image of the adjoint line bundle}

\author[I.Biswas]{Indranil Biswas}

\address{School of Mathematics, Tata Institute of Fundamental Research, 1 Homi Bhabha Road,
Colaba, Mumbai 400005}

\email{indranil@math.tifr.res.in}

\author[F. Laytimi]{Fatima Laytimi}

\address {Math\'ematiques - b\^{a}t. M2, Universit\'e Lille 1,
F-59655 Villeneuve d'Ascq Cedex, France}

\email {fatima.laytimi@univ-lille.fr}

\author[D. S. Nagaraj]{D. S. Nagaraj}

\address {Indian Institute of Science Education and Research, Tirupati}

\email {dsn@iisertirupati.ac.in}

\author[W. Nahm]{Werner Nahm}

\address{Dublin Institute for Advanced Studies, 10 Burlington Road,
Dublin 4, Ireland}

\email{wnahm@stp.dias.ie}



\date{}

\begin{abstract} 
We give an algebro-geometric proof of the fact that for a smooth fibration $\pi\,:\,X\,\longrightarrow\, Y$ of projective varieties,
the direct image $\pi_*(L\otimes K_{X/Y})$ of the adjoint line bundle
of an ample (respectively, nef and $\pi$-strongly big) line bundle $L$ is ample (respectively, nef and big).
\end{abstract}

\maketitle

\section{Introduction} \setcounter{page}{1}

Throughout we work over the field $\mathbb{C}$ of complex numbers.
For the standard notation used here the reader is referred to \cite{La1}, \cite{La2}.

We first give an algebraic proof of the following: 

\begin{thm}\label{main1}
Let $\pi\,:\,X\,\longrightarrow\, Y$ be a smooth fibration of smooth projective varieties
and $K_{X/Y}$ the relative canonical line bundle for it. For any ample line bundle $L$ on $X$
the direct image $\pi_* (L\otimes K_{X/Y})$ is either zero or an ample vector bundle. 
\end{thm}

Mourougane in \cite{Mour} has proved by analytic methods that the vector bundle $\pi_* (L\otimes K_{X/Y})$ is Griffith positive (hence ample).

We define below the notion of $\pi$-strongly big line bundle (see, Definition \ref{sbig}) and prove the following:

\begin{thm}\label{main2}
Let $\pi\,:\,X\,\longrightarrow\, Y$ be a smooth fibration of projective varieties.
If L is a $\pi$-strongly big line bundle on $X,$ then the vector bundle $\pi_* (L\otimes K_{X/Y})$
is big. Moreover, if L is nef and $\pi$-strongly big, then
$\pi_* (L\otimes K_{X/Y})$ is also nef and big.
\end{thm}

\section{Proof of Theorem \ref{main1}}

We need to recall following known results.

\begin{thm}[{\cite[Theorem 4.1.10]{La1}}]\label{BG}
Let $\xi$ be a line bundle on a projective variety $Y$. Fix a positive integer $d$. Then there exists a projective variety $\widetilde Y$,
a finite surjective morphism $f\,:\,\widetilde{Y} \,\longrightarrow\, Y$ and a line bundle $M$ on $\widetilde{Y}$, such that $f^*\xi
\,\simeq\, M^{\otimes d}.$
\end{thm}
 
\begin{thm}[{\cite[Theorem 2.43]{Vie}}]\label{Vie} 
Let $\pi\,:\,X\,\longrightarrow\, Y$ be a smooth fibration of projective varieties.
If $L$ is semi-ample, then $\pi_*(L\otimes K_{X/Y})$ is nef.
\end{thm}

Take $\pi$ as above.
If $L$ is an ample bundle on $X$, then the restriction $L\big\vert_{X_y}$ of $L$ to the fiber $X_y\,:=\, \pi^{-1}(y)$ over
$y\,\in\, Y$ is ample, and hence from Kodaira vanishing theorem it follows that $\pi_* (L\otimes K_{X/Y})$ is either $0$
or a vector bundle on $Y.$

Now we start the proof Theorem \ref{main1}. Let $L$ be an ample line bundle on $X$ and $\xi$ a line bundle on $Y.$ 
Then there is an integer $d\,>\,0$ such that $L^d \otimes \pi^*\xi^*$ is ample, where
$\xi^*$ is the dual of the line bundle $\xi$. We assume that $\xi$ is ample. 
By Theorem \ref{BG} there exists a projective variety $\widetilde{Y},$ a finite surjective morphism $f\,:\,\widetilde{Y}\,
\longrightarrow\, Y$ and a line bundle $M$ on $\widetilde{Y},$ such that
\begin{equation}\label{t1}
f^*\xi\,= \,M^{\otimes d}.
\end{equation}
Note that $f^*\xi$ is ample because $f$ is a finite map and $\xi$ is ample. Hence from \eqref{t1} it follows that $M$ is ample
(recall that $d\, >\, 0$). We have a commutative diagram:
\begin{equation}\label{e1}
\begin{CD}
	\widetilde{X} @>\widetilde{f}>> X \\
	@VV\widetilde{\pi} V @VV \pi V\\
	\widetilde{Y} @>f>> Y
\end{CD}
\end{equation}
The line bundle $L^d \otimes \pi^*\xi^*$ is ample if and only if $ \widetilde f^*( L^d \otimes \pi^*\xi^*) $ is ample,
because $ \widetilde f$ is finite. We have
\[ \begin{array}{rcl}
{\widetilde f}^*(L^d \otimes \pi^*\xi^*) &= &({\widetilde f}^*L^d)\otimes ({\widetilde f}^*{\pi^*}{\xi}^*)\\
{} &=& (\widetilde{f}^*L^d) \otimes (\widetilde{\pi}^*{f^*}\xi^*)\\
{}&=& (\widetilde{f}^*L^d)\otimes (\widetilde{\pi}^* ({M^*})^d) \\
{}& = & ((\widetilde{f}^*L)\otimes (\widetilde{\pi}^* M^*))^d.
\end{array} \]
Set $\mathcal{L}\,=\,\widetilde{f}^*L\otimes \widetilde{\pi}^* M^*;$
then $\mathcal L$ is ample.

{}From Theorem \ref{Vie} we know that $\widetilde{\pi}_*(\mathcal {L}\otimes K_{\widetilde{X}/\widetilde{Y}})$ is nef. Using projection formula,
$$\widetilde{\pi}_*(\mathcal{L}\otimes K_{\widetilde{X}/\widetilde{Y}})\,=\, \widetilde{\pi}_*(\widetilde{f}^{*}{L}\otimes
K_{\widetilde{X}/\widetilde{Y}})\otimes M^*.$$
The vector bundle $\widetilde{\pi}_*((\widetilde{f}^*{L})\otimes K_{\widetilde{X}/\widetilde{Y}})\,=\, \widetilde{\pi}_*((\widetilde{f}^*{L})\otimes
K_{\widetilde{X}/\widetilde{Y}})\otimes M^* \otimes M$ is ample, since the tensor product of a nef vector bundle
and an ample vector bundle is ample (see \cite[Lemma 1.3]{FL}). 

But $\widetilde{\pi}_*((\widetilde{f}^*{L})\otimes K_{\widetilde{X}/\widetilde{Y}})\,=\, f^*(\pi_*(L\otimes K_{X/Y}))$ by the commutative diagram
in \eqref{e1}. Hence $\pi_*(L\otimes K_{X/Y})$ is ample. This completes the proof of Theorem \ref{main1}.

\begin{remark}
For any ample line bundle $L$ on a smooth projective variety $Z$, the vanishing theorem of Kodaira says
that $H^i(Z,\, L\otimes K_Z)\,=\, 0$ for all $i\, \geq\,1$, where $K_Z$ is the canonical line
bundle of $Z$ \cite{La1}. Let $\pi\,:\,X\,\longrightarrow\, Y$ be a smooth fibration of smooth projective
varieties. Let $L$ be an ample line bundle on $X$. From Kodaira vanishing theorem it now follows that
$R^i\pi_* (L\otimes K_{X/Y})\,=\, 0$ for all $i\,\geq\, 1$.
\end{remark}

\begin{remark}
In Theorem \ref{main1} set $X\,=\, {\mathbb P}(E)$, where $E$ is a vector bundle on $Y$. Let $L$ be an
ample line bundle on ${\mathbb P}(E)$. Assume that $\pi_*(L\otimes K_{X/Y})\,\not=\, 0$. Then there is
a unique positive integer $d$ and a unique line bundle $L_0$ on $Y$ such that
$$
L\otimes K_{X/Y}\,=\, {\mathcal O}_{{\mathbb P}(E)}(d)\otimes \pi^*L_0.
$$
Hence using the projection formula, we have
$$
\pi_*(L\otimes K_{X/Y})\,=\, \pi_*({\mathcal O}_{{\mathbb P}(E)}(d)\otimes \pi^*L_0)\,=\,
\text{Sym}^d(E)\otimes L_0,
$$
where $\text{Sym}^d(E)$ is the $d$-th symmetric product of $E$. Now Theorem \ref{main1} says that
$\text{Sym}^d(E)\otimes L_0$ is ample.

When $Y$ is a smooth projective curve, $\text{Sym}^d(E)\otimes L_0$ is ample if and only if
$\mu_{\rm min}(\text{Sym}^d(E)\otimes L_0)\, >\, 0$ \cite[p.~84, Theorem 2.4]{Ha}.
\end{remark}

\section{Proof of Theorem \ref{main2}}
 
\begin{Def}\label{sbig}
Let $\pi\,:\,X\,\longrightarrow\, Y$
be a smooth fibration of smooth projective varieties. A line bundle $L$ on $X$ is said to be {\em $\pi$-strongly big} if there
is an effective divisor $D$ on $Y$ with simple normal crossing support (see \cite[Definition 4.1.1]{La1})
such that $L^{d}\otimes \pi^{*}(\OOO_Y(D))^*$ is ample for some integer
$d\,>\,0,$ where $\pi^{*}(\OOO_Y(D))^*\,=\, \pi^{*}(\OOO_Y(-D)).$
\end{Def}

The following proposition states some properties of $\pi$-strongly big line bundles.

\begin{prop}\label{prop1} Take $\pi$ and $L$ as in Definition \ref{sbig}. If $L$ is $\pi$-strongly big, then it is big
and $\pi$-ample. In addition, if $L$ is nef, then $\pi_*(L\otimes K_{X/Y})$ is also nef.
\end{prop}

\begin{proof}
Since $\pi^*(\OOO_Y(D)) \,= \,\OOO_X(\pi^{-1}(D))$, and 
$\pi^{-1}(D)$ is effective, it follows from a
characterization of big divisors (see \cite[p.~141, Corollary 2.2.7]{La1}) that $L$ is big. 
On the other hand, by definition, and the fact that 
$L^d(-\pi^{*}\OOO_Y(D))\big\vert_{X_y} \,\simeq\, L^d\big\vert_{X_y}$
for all $y \,\in\, Y,$ we obtain that $L$ is $\pi$-ample; here $X_y$ denotes the fiber of $\pi$ over $y.$ 
With these observations, the second statement of the proposition follows from \cite[Theorem 1]{Mour}.
\end{proof}

Before proving Theorem \ref{main2} we need the following lemma
which is a consequence Kawamata coverings result (see \cite[Proposition 4.1.12]{La1}).

\begin{lem}\label{Kaw}
In Theorem \ref{BG} assume $Y$ is non-singular and $\xi \,=\, \OOO_Y(D),$ where $D$ is an effective divisor on $Y$ with
simple normal crossing support (i.e., the reduced divisor $D_{\rm red}$ is normal crossing divisor). Then
given any integer $d \,>\, 0$, there is a finite surjective morphism
$f\,:\,\widetilde{Y} \,\longrightarrow\, Y, $ and a line bundle $M\,=\,\OOO_{\widetilde Y}(\widetilde{D})$ on $\widetilde{Y}$,
such that $f^*\xi\,\simeq\, M^d,$ and $\widetilde{D}$ is effective.
\end{lem}

For the proof of Theorem \ref{main2} we proceed as in the proof of Theorem \ref{main1}. Let $L$ be a $\pi$-strongly big line bundle on $X$. 
Then there exists an effective divisor $D$ on $Y$ with simple normal crossing support and an integer $d\,>\,0,$, such that $L^d \otimes
\pi^*{\OOO_Y (D)}^*$ is an ample line bundle. By Lemma \ref{Kaw} there exists a projective variety $\widetilde Y,$ a finite surjective
morphism $f\,:\,\widetilde{Y} \,\longrightarrow\, Y$ and a line bundle $M\,=\,\OOO_{\widetilde Y}(\widetilde{D})$ with $\widetilde D$
effective, such that $f^*(\OOO_Y(D))\,=\, M^d$. We have the commutative diagram:
\begin{equation}\label{e2}
\begin{CD}
\widetilde{X} @>\widetilde{f}>> X \\
@VV\widetilde{\pi} V @VV \pi V\\
\widetilde{Y} @>f>> Y
\end{CD}
\end{equation}
The line bundle $L^d \otimes \pi^*(\OOO_Y{(D)})^* $ is ample if and only if $\widetilde{f}^*( L^d \otimes \pi^*(\OOO_Y{(D)})^*) $ is ample. 
We have
\[ \begin{array}{rcl}
	{\widetilde f}^*(L^d \otimes \pi^*(\OOO_Y{(D)})^*)&\,=\, &({\widetilde f}^*L^d)\otimes ({\widetilde f}^*\pi^*(\OOO_Y{(D)})^*)\\
{}	&=\,& (\widetilde{f}^*L^d)\otimes (\widetilde{\pi} ^*(f^*\OOO_Y{(D)})^*) \\
{} &=\,& (\widetilde{f}^*L^d)\otimes (\widetilde{\pi}^* {M^*})^d \\
{} &=\,& ((\widetilde{f}^*L)\otimes (\widetilde{\pi}^* M^*))^d.
\end{array}
\]
Set $ \mathcal {L}\,=\, (\widetilde f^*L)\otimes \widetilde{\pi}^* M^*$.
Then $\mathcal L$ is ample, and hence from Theorem \ref{main1} it follows that $\widetilde{\pi}_*(\mathcal {L}\otimes
K_{\widetilde{X}/{\widetilde Y}})$ is ample. Now
$$\widetilde{\pi}_*(\mathcal{L}\otimes K_{\widetilde{X}/\widetilde{Y}})\,=\,\widetilde{\pi}_*((\widetilde{f}^{*}L)\otimes
K_{\widetilde{X}/\widetilde{Y}})\otimes M^*,$$ where $M\,=\, \OOO_{\widetilde Y}(\widetilde D)$ is effective. Then 
$\widetilde{\pi}_*((\widetilde{f}^{*}L)\otimes K_{\widetilde{X}/\widetilde{Y}})$ is big,
because an ample vector bundle tensored with an effective line bundle is big
(see \cite[Example 6.1.23]{La2}). 

But $\widetilde{\pi}_*(\mathcal{L}\otimes K_{\widetilde{X}/\widetilde Y})\,=\, f^*(\pi_*(L\otimes K_{X/Y}))$ by
\eqref{e2}. Since $f$ is finite we conclude that $\pi_*(L\otimes K_{X/Y})$ is big.

Moreover, if $L$ is $\pi$-strongly big and nef, by Proposition \ref{prop1} the direct image
$\pi_*(\mathcal{L}\otimes K_{\widetilde{X}/\widetilde{Y}})$ is nef. 
This completes the proof.

\section{Some observations}

\begin{prop}\label{prop2}
Let $L$ be a line bundle on a non-singular projective variety $X$ and $E$ a vector bundle on $X.$ Assume
that the vector bundle $E^{\otimes s}\otimes L$ is generated by its global sections for every integer $s\,>\!\!> \,0$.
Then $E$ is nef.
\end{prop}

\begin{proof}
Assume that $E$ is not nef. Then a criterion for nef
bundles (see \cite[Proposition 6.1.18]{La2}) gives that
there is a non-constant morphism $f \,:\,C \,\longrightarrow\, X$ from
a non-singular projective curve $C$, and a line bundle $M$ on $C$ of degree
$ n \,<\, 0$, such that there is a surjective map $$ f^*E \,\longrightarrow\, M.$$
So for all $s \,>\, 0$ there is a surjection 
$$(f^*E)^{\otimes s} \otimes f^{*}L\,=\, f^*(E^{\otimes s} \otimes L)\,\longrightarrow\, M^{\otimes s} \otimes (f^{*}L).$$

Note that the degree of the line bundle $M^{\otimes s} \otimes (f^{*}L)$
is $n.s + d,$ where $d$ is the degree of the line bundle $f^{*}L.$ Since $n\,<\,0$, we have $n.s + d \,<\, 0$ for all large $s.$
The bundle $(f^*E)^{\otimes s} \otimes (f^{*}L)$ is generated by its global sections, by hypothesis, for
all large $s.$ Hence it cannot admit a surjective map to a negative line bundle. 
Thus we get a contradiction to the assumption that $E$ is not nef.
This proves the proposition.
\end{proof}

We give an example to show that the assumptions do hold.

\begin{example}\label{ex1}
Let $Z$ and $Y$ be two smooth projective varieties, and let $X \,=\, Y \times Z.$ Let $\pi_1\,:\, X 
\,\longrightarrow\, Y$ and $\pi_2\,: \,X\,\longrightarrow\, Z$ be the natural projections. If $L_1$ is a big 
line bundle on $Y$, and $L_2$ is an ample line bundle on $Z$, then the line bundle $M\,=\,(\pi_1^*L_1) 
\otimes (\pi_2^*L_2)$ is big on $X$ and $(\pi_1)_*(M\otimes K_{X/Y} )\,=\, {\text H}^0(Z,\,L_2\otimes 
K_Z)\otimes L_1$ is big provided ${\text H}^0(Z,\,L_2\otimes K_Z) \,\neq\, 0$.
\end{example}

\begin{remark}
Note that in Example \ref{ex1} the line bundle $M$ is $\pi_1$-strongly big
provided the big bundle $L_1$ satisfies the following property: There is an integer $d\, >\, 0$, and a divisor
$D$ on $Y$ with a simple normal crossing support, such that $L^d(-D)$ is ample on $Y.$ 
\end{remark}

Let $\pi\,:\, X\,\longrightarrow\, Y$ be a smooth fibration of smooth projective varieties.
Next example shows that when $L$ is nef and $\pi$-big, the direct image $\pi_* (L\otimes K_{X/Y})$ is not necessarily big. 

\begin{example}\label{ex2}
Let $Z$ and $Y$ be two smooth projective varieties, and let $X \,=\, Y \times Z.$ Let $\pi_1\,:\, X \,\longrightarrow\, Y$ and
$\pi_2\,:\, X \, \longrightarrow\, Z$ be the natural projections. 
Let $M$ be an ample line bundle on $Z$. Then the line bundle $ \pi_2^* M$ is nef and
$\pi_1$-ample (and hence $\pi_1$-big) on $X $ and 
$(\pi_1)_*((\pi_2^*M)\otimes K_{X/Y} )\,=\, {\text H}^0(Z,\,M\otimes K_Z)\otimes \OOO_Y$ is a
trivial bundle, if ${\text H}^0(Z,\,L_2\otimes K_Z) \,\neq\, 0$, which is only nef but not big. 
\end{example}

\begin{Def}\label{bbig}
Let $\pi\,:\,X\,\longrightarrow\, Y$
be a smooth fibration of smooth projective varieties. A line bundle $L$ on $X$ is said to be {\em $\pi$-weakly big} if there
is an effective divisor $D$ on $Y$ 
such that $L^{d}\otimes f^{*}(\OOO_Y(D))^*$ is ample for some integer
$d\,>\,0,$ where $f^{*}(\OOO_Y(D))^*\,=\, f^{*}(\OOO_Y(-D)).$
\end{Def}

\begin{remark}
For a $\pi\,:\,X\,\longrightarrow\, Y$
smooth fibration of smooth projective varieties a $\pi$-strongly big line bundle 
is $\pi$-weakly big but is $\pi$-weakly big then it may not be
$\pi$-strongly big.
\end{remark}

\begin{prop}
Let $\pi\,:\,X\,\longrightarrow\, Y$ as in Theorem \ref{main1}
and $L$ be a $\pi$-weakly big line bundle on $X.$ Then there is an integer $r > 0$
such that the bundles $\pi_*(L^{rm}\otimes K_{X/Y})$ is big for all $m > 0.$
\end{prop}

{\bf Proof:} By definition of $\pi$-weakly big line bundle there is a effective divisor
$D$ on $Y$ and a positive integer $r$ such that $L^r\otimes f^{*}(\OOO_Y(D))^*$ is 
an ample and hence $L^{rm}\otimes f^{*}(\OOO_Y(mD))^*$ is ample for all $m>0.$
By Theorem \ref{main1} $\pi_*(L^{rm}\otimes f^{*}(\OOO_Y(mD))^*\otimes K_{X/Y})$
is ample for all $m \,>\, 0$. But by projection formula
$$\pi_*(L^{rm}\otimes\otimes K_{X/Y}) \,=\, \pi_*(L^{rm}\otimes f^{*}(\OOO_Y(mD))^*\otimes K_{X/Y})\otimes \OOO_Y(mD)$$
for all $m > 0.$ Thus $\pi_*(L^{rm}\otimes\otimes K_{X/Y})$ is big for $m>0.$

{\bf Conjecture:}
Let $\pi\,:\,X\,\longrightarrow\, Y$ as in Theorem \ref{main1} 
and $L$ be a big and nef line bundle on $X.$ Then $\pi_*(L\otimes K_{X/Y})$
is nef and big.

\bigskip

\section*{Acknowledgements}

DSN thanks CEMPI of Lille university for the financial support.

\end{document}